\documentclass[12pt,notitlepage]{amsart}
\usepackage{latexsym,amsfonts,amssymb,amsmath,amsthm} 
\newcommand{\mz}{\ensuremath{\mathbb Z}}

\newcommand{\mq}{\ensuremath{\mathbb Q}}
\newcommand{\mc}{\ensuremath{\mathbb C}}
\newcommand{\mf}{\ensuremath{\mathbb F}}

\newcommand{\mymod}{\ensuremath{\negthickspace \negmedspace \pmod}}
\newcommand{\shortmod}{\ensuremath{\negthickspace \negthickspace \negthickspace \pmod}}

\newcommand{\onehalf}{\ensuremath{ \frac{1}{2}}}
\newcommand{\half}{\ensuremath{ \frac{1}{2}}}
\newcommand{\thalf}{\tfrac{1}{2}}
\newcommand{\notdiv}{\ensuremath{\nmid}}

\theoremstyle{plain}		
	\newtheorem{mytheo}{Theorem}[section]
	
	\newtheorem{myprop}[mytheo]{Proposition}
	\newtheorem{mycoro}[mytheo]{Corollary}
     \newtheorem{mylemma}[mytheo]{Lemma}
	\newtheorem{mydefi}[mytheo]{Definition}
	\newtheorem{myconj}[mytheo]{Conjecture}

\theoremstyle{remark}

\numberwithin{equation}{section}
\begin{document}

\title[Moments of critical values of elliptic curves]{Moments of the critical values of families of elliptic curves, with applications}
\today
\author{Matthew P. Young} 
\address{Department of Mathematics
Texas A\&M University,
College Station, TX 77843-3368 }
\email{myoung@math.tamu.edu}
\thanks{This research was supported by an NSF Mathematical Sciences Post-Doctoral Fellowship and by the American Institute of Mathematics.}
\begin{abstract}
We make conjectures on the moments of the central values of the family of all elliptic curves and on the moments of the first derivative of the central values of a large family of positive rank curves.  In both cases the order of magnitude is the same as that of the moments of the central values of an orthogonal family of $L$-functions.  Notably, we predict that the critical values of all rank $1$ elliptic curves is logarithmically larger than the rank $1$ curves in the positive rank family.  

Furthermore, as arithmetical applications we make a conjecture on the distribution of $a_p$'s amongst all rank $2$ elliptic curves, and also show how the Riemann hypothesis can be deduced from sufficient knowledge of the first moment of the positive rank family (based on an idea of Iwaniec).
\end{abstract}
\maketitle

\section{Introduction}
Recently there have been many advances in the study of ranks of elliptic curves arising from random matrix theory.  For instance, Conrey {\em et al} \cite{CKRS} have studied many interesting statistics of the family of quadratic twists of a fixed elliptic curve.  In particular, they make a precise conjecture on the relative frequency of quadratic twists of rank two where the comparison is between the sets of curves twisted by integers that are quadratic residues $\pmod{p}$ and those that are quadratic nonresidues $\pmod{p}$.  This conjecture is deduced from a general moment conjecture on the central values of the families of interest \cite{CFKRS} combined with random matrix theory heuristics developed by Keating and Snaith \cite{KS}.

In this paper we study the analogous problems for the family of all rational elliptic curves.  That is, we make a conjecture on the moments of the central values of this family, using the general recipe presented in \cite{CFKRS}.  

In general such a conjecture for a family of $L$-functions has its gross order of magnitude determined only by the symmetry type of the family.  For example, for an orthogonal family $\mathcal{F}$ the general conjecture is
\begin{equation}
\frac{1}{|\mathcal{F}(X)|} \sum_{f \in \mathcal{F}(X)} L(\thalf, f)^k \sim a_k g_k (\log{X})^{\frac{k(k-1)}{2}},
\end{equation}
where it is understood that $\mathcal{F}(X)$ is a subset of $\mathcal{F}$ with conductors $\ll X$ and the asymptotic holds as $X \rightarrow \infty$.  Here $a_k$ is called the arithmetical factor and $g_k$ is a constant arising from random matrix theory.  

We also study the moments of the first derivative of the $L$-functions at the central point for a positive rank family $\mathcal{F}'$.  It is perhaps not obvious what to expect for a family $\mathcal{F}'$ where the central values should typically vanish to order one or two (depending on the sign in the functional equation).  We predict that the order of magnitude of the $k$-th moment of $L'(\thalf, E)$ where $E$ ranges over $\mathcal{F}'$ is the same as that given above for an orthogonal family.  This conjecture lends evidence to the idea that the family $\mathcal{F}'$ should be modeled by an orthogonal family with the caveat that one should `add' one zero to the central point (the `independent' model; see Conjecture 1.1 of \cite{Miller} and the article \cite{Farmer}). 
 

As an application of the moment conjecture we predict the relative proportion of rank $2$ elliptic curves whose coefficients $a$ and $b$ of the Weierstrass equation $y^2 = x^3 + ax + b$ lie in prescribed arithmetic progressions $\pmod{q}$, similarly to the work of \cite{CKRS}.  This gives a large number of conjectures that have the attractive feature of potentially being tested numerically, based on the strikingly good agreement that was seen with the quadratic twist families considered by \cite{CKRS}.  

The families of elliptic curves investigated in this paper (especially the positive rank family) can be thought of as somewhat exotic tests of the general moment conjectures of \cite{CFKRS}.

The arithmetical constant for the family of all elliptic curves is more subtle than for other families previously considered as it depends on the traces of the Hecke operators acting on the space of weight $k$ cusp forms for the full modular group.  In previous examples (Riemann zeta, families of Dirichlet $L$-functions, weight $k$ level $N$ newforms, to name a few), the arithmetical factor was essentially given in terms of rational functions in $p$.  A large part of this paper is the computation of the arithmetical factor for our family.  The key to this computation is a useful formula for the orthogonality relation for the family of all elliptic curves, which we compute with Proposition \ref{prop:Q*formula}; the arithmetical factor is essentially the Dirichlet series constructed from the orthogonality relation.  In the case of the positive rank family there does not appear to be as nice a formula for the arithmetical factor as there is for the family of all elliptic curves.

The methods of this paper can be easily modified to obtain similar moment conjectures for other families of elliptic curves.  However, the computation of the arithmetical factor in terms of easily computable `extrinsic' (non-tautological) quantities is not easily generalized to other families (see the remarks after Conjecture \ref{conj:momentfirstterm'} for a more precise discussion of what is meant here).

There are a variety of ways to order elliptic curves: by conductor, by absolute, minimal discriminant, or by taking coefficients in the Weierstrass equation to lie in a box.  Furthermore, there is the question of whether to count by isomorphism class or by isogeny class (put another way: is the family composed of curves or by $L$-functions?).  However, there is reason to believe that almost every isogeny class contains only one isomorphism class; Watkins briefly touches on this issue (\cite{Watkins}, section 5).  

We have ordered our curves by taking the coefficients to lie in a box for a practical reason: it is possible to do explicit computations with this ordering.  It may be most natural to order curves by conductor, but it is difficult to work with this ordering.  Recently, Watkins \cite{Watkins} has developed various heuristics that, amongst other things, allows one to get some handle on the ordering by conductor by way of the ordering in boxes.  It would be interesting to compute the orthogonality relation for the family of elliptic curves ordered by conductor.  Ordering by boxes is particularly pleasant because of periodicity of the Dirichlet series coefficients.

\subsection{Notation and definitions}
Let $E_{a,b}$ be the elliptic curve over $\mq$ given by the Weierstrass equation
\begin{equation}
\label{eq:weierstrass}
E_{a,b}: y^2 = x^3 + ax + b,
\end{equation}
with discriminant $\Delta  =\Delta_{a,b}= -16(4a^3 + 27b^2) \neq 0$ and conductor $N$.
For integers $r$, $t$, and squarefree $q$ coprime with $6$, and parameter $X >0$ we take the family $\mathcal{F}(X) = \mathcal{F}_{r,t}(X)$ defined by
\begin{equation}
\mathcal{F}(X)  = \{E_{a,b} : 
a \equiv r \mymod{6q}, \;
b \equiv t \mymod{6q}, \;
|a| \leq X^{1/3}, \;
|b| \leq X^{1/2}, \;
p^4 | a \Rightarrow p^6 \nmid b \}.
\end{equation}
We also suppose $(4r^3 + 27t^2, 6q) = 1$, so in particular $(3,r) = (2,t)=1$, and $(\Delta_{a,b},3q) = 1$ for $E_{a,b} \in \mathcal{F}(X)$.  In Section \ref{section:background} we review some basic facts about elliptic curves, and develop some of the properties of the curves in the family $\mathcal{F}(X)$.  Occasionally we write $\mathcal{F}^+$ to denote the set of $E \in \mathcal{F}$ with root number $w_E = +1$.

Our positive rank family $\mathcal{F}'$ is defined by
\begin{equation}
\mathcal{F}'(X) = \{E_{a,b^2} : 
a \equiv r \mymod{6}, \;
b \equiv t \mymod{6}, \;
|a| \leq X^{1/3}, \;
|b| \leq X^{1/4}, \;
p^4 | a \Rightarrow p^3 \nmid b \}.
\end{equation}
We could also take $q$, $r$, and $t$ as in the definition of $\mathcal{F}$ to analyze the behavior of $a$ and $b$ in arithmetic progressions, but have taken $q=1$ for simplicity.

Each curve $E_{a,b^2}$ has the point $(0, b)$ which is almost always of infinite order (see Theorem \ref{theorem:L-N} and subsequent remarks).
The Birch and Swinnerton-Dyer conjecture therefore predicts that $L(\thalf, E_{a,b^2}) = 0$ for almost all $a$ and $b$.
In addition, the sign in the functional equation for this family is expected to be evenly distributed between $\pm 1$ (see Proposition \ref{prop:rootnumber} below).  

Let $G(s)$ be the Barnes $G$-function, which satisfies $G(1) = 1$ and $G(s +1) = \Gamma(s) G(s)$.  The $k$-th Chebyshev polynomial of the second kind is denoted by $U_k$ and $Tr_{l}(p)$ is the trace of the Hecke operator $T_p$ acting on the space of weight $l$ cusp forms on the full modular group.  We let $Tr_{l}^*(p)$ be the `scaled' trace determined by $Tr_l(p) = p^{\frac{l-1}{2}} Tr_{l}^*(p)$. We let $d\mu_{ST}$ be the Sato-Tate measure, i.e.
\begin{equation}
\int f \, d\mu_{ST} := \frac{2}{\pi} \int_0^{\pi} f(\theta) \sin^2\theta d\theta.
\end{equation}

\subsection{Moment conjectures for the family of all elliptic curves}
We now state
\begin{myconj} 
\label{conj:momentfirstterm}
For any $k \in \mc$ such that $\text{Re } k > -\half$,
\begin{equation}
\label{eq:momentfirstterm2}
\frac{1}{|\mathcal{F}(X)|} \sum_{E \in \mathcal{F}(X)} L(\thalf, E)^k \sim \thalf a_k g_k (\log{X})^{\frac{k(k-1)}{2}}
\end{equation}
holds as $X \rightarrow \infty$, where
\begin{equation}
g_k = 2^{k/2} \frac{ G(1+k) \sqrt{\Gamma(1+2k)}}{\sqrt{G(1 + 2k) \Gamma(1+k)}}
\end{equation}
is a certain constant familiar from random matrix theory, and $a_k$ is an arithmetical factor given by an explicit absolutely convergent Euler product (for which see \eqref{eq:ak}).

\end{myconj}
Remarks.  We use the convention that $0^k = 0$ for any $k$ (alternatively, one could only sum over nonzero central values).
The restriction to $\text{Re } k > -\half$  arises because the Barnes $G$-function has its rightmost pole at $k=-\half$.  

For integral $k \geq 1$ this conjecture is a special case of the more precise Conjecture \ref{conj:Pkconj}.   Conjecture \ref{conj:momentfirstterm} is somewhat simpler and also provides an analytic continuation of $a_k$ to complex $k$ which is a necessary ingredient for deriving Conjecture \ref{conj:rank2}.  By taking $k=0$ and computing that $a_0 = g_0 = 1$ we obtain
\begin{mycoro}
Conjecture \ref{conj:momentfirstterm} implies the average rank of the family of all elliptic curves $\mathcal{F}$ is $\half$.
\end{mycoro}
We present this corollary simply to illustrate the strength of Conjecture \ref{conj:momentfirstterm} and the usefulness of extending the formulas to more general $k$ than positive integers.  This result also indicates that it will be difficult to check Conjecture \ref{conj:momentfirstterm} numerically, because of the well-known disparity between the expected proportion of rank $2$ elliptic curves and the numerical evidence; see \cite{BMSW} for a recent survey on this fascinating problem.  It may be that the large value distribution converges more quickly, so that there may be better numerical agreement for (slightly) larger $k$; this deserves investigation.  In any case, one should include all lower-order terms in the conjectured asymptotic.  

Our most precise moment conjecture is as follows.
\begin{myconj}
\label{conj:Pkconj}
Let $k$ be a nonnegative integer.  Then for some $\delta > 0$,
\begin{equation}
\sum_{E \in \mathcal{F}(X)} L(\thalf, E)^k 
= \half \sum_{E \in \mathcal{F}(X)} P_k(N_E) ( 1 + O(N_E^{-\delta})),
\end{equation}
where 
\begin{multline}
\label{eq:PkN}
P_k(N) = \\
\frac{(-1)^{\frac{k(k-1)}{2}} 2^{k}}{k! } \frac{1}{(2\pi i)^k} \oint \cdots \oint H(z_1, \ldots, z_k) 
\frac{\Delta(z_1^2, \ldots, z_k^2)^2}{\prod_{i=1}^k z_i^{2k-1}} 
\prod_{i = 1}^k X_N^{-\half}(\thalf + z_i)
dz_1 \cdots dz_k,
\end{multline}
and $X_N(s) = X_E(s)$ is such that $L(s,E) = w_E X_E(s) L(1-s,E)$ (see \eqref{eq:X}). Here
\begin{equation}
H(z_1, \ldots, z_k) = A_k (z_1, \ldots, z_k) \prod_{1 \leq i < j \leq k} \zeta(1 + z_i + z_j),
\end{equation}
where the arithmetical factor $A_k$ is holomorphic and nonzero in a neighborhood of $(0, \ldots, 0)$.  Here $P_k(N)$ is a polynomial in $\log{N}$ of degree $k(k-1)/2$.

\end{myconj}
Of course this formula is analogous to Conjectures 1.5.3 and 1.5.5 of \cite{CFKRS}, the only essential difference being the exact form of the arithmetical factor.  Due to the size of the formulas we have delayed the precise formulation of the arithmetical factors; see Proposition \ref{prop:Ak}.

The factor $\half$ appears because roughly half of the $L$-functions vanish since the root number is $-1$. 

\subsection{Moment conjectures for the positive rank family}
\begin{myconj} 
\label{conj:momentfirstterm'}
For any $k \in \mc$ such that $\text{Re } k > -\half$ there exists $a_k' \neq 0$ such that
\begin{equation}
\frac{1}{|\mathcal{F}'(X)|} \sum_{E \in \mathcal{F}'(X)} (L'(1/2, E))^k \sim \thalf a_k' g_k (\log{X})^{\frac{k(k-1)}{2}}
\end{equation}
holds as $X \rightarrow \infty$
\end{myconj}

It is possible to write a formula for $a_k'$ as an Euler product but it involves the sums $Q_{\square}^*(p^{e_1}, \ldots, p^{e_k})$ discussed in Section \ref{section:ak'}.  This is in contrast to the family $\mathcal{F}$ where we have evaluated similar sums in terms of Chebyshev polynomials and the traces of the Hecke operators on $\Gamma(1)$.

Our most precise moment conjecture is as follows.
\begin{myconj}
\label{conj:Qkconj}
Let $k$ be a nonnegative integer.  Then for some $\delta > 0$,
\begin{equation}
\sum_{E \in \mathcal{F}'(X)} (L'(1/2, E))^k 
= \half \sum_{E \in \mathcal{F}'(X)} Q_k(N_E) ( 1 + O(N_E^{-\delta})),
\end{equation}
where $Q_k$ has the form
\begin{multline}
\label{eq:QkN}
 Q_k(N) = \\
\frac{(-1)^{\frac{k(k-1)}{2}} 2^{k}}{k! } \frac{1}{(2\pi i)^k} \oint \cdots \oint H(z_1, \ldots, z_k) 
\frac{\Delta(z_1^2, \ldots, z_k^2)^2}{\prod_{i=1}^k z_i^{2k}} 
\prod_{i = 1}^k X_N^{-\half}(\thalf + z_i)
dz_1 \cdots dz_k 
\end{multline}
and
\begin{equation}
H(z_1, \ldots, z_k) = A_k' (z_1, \ldots, z_k) \frac{\prod_{1 \leq i < j \leq k} \zeta(1 + z_i + z_j)}{\prod_{1 \leq i \leq k} \zeta(1 + z_i)},
\end{equation}
where the arithmetical factor $A_k'$ is given by an Euler product that is absolutely convergent in a neighborhood of $(0, \ldots, 0)$.
\end{myconj}
In terms of the behavior near the origin there are two essential differences between Conjectures \ref{conj:Pkconj} and \ref{conj:Qkconj}.  Namely, in \eqref{eq:QkN} there is the extra factor $\prod_i z_i^{-1}$ and the function $H(z_1, \ldots, z_k)$ has the additional factor $\prod_i \zeta^{-1}(1 + z_i)$.  Thus $Q_k(N)$ and $P_k(N)$ have the same degree, because the polar behavior near the origin in their integral representations are the same.  The factor $\prod_i z_i^{-1}$ arises from the differentiation; the extra zeta function factors arise from the positive rank of the family $\mathcal{F}'$.

\subsection{The relative frequency of rank $2$ and higher curves}
In this section we consider the question of the distribution of $a_p$'s amongst all rank $2$ elliptic curves.  The distribution of $a_p$'s amongst all elliptic curves is known from \cite{Birch}; see also \cite{Schoof}.  It is expected that rank $2$ curves have $a_p$'s that are biased towards being negative.  The conjecture in this section gives a precise prediction of this bias.

Given $r$ and $t \pmod{p}$, $\lambda_{a,b}(p)$ is fixed for $a \equiv r$, $b \equiv t \pmod{p}$.  The number of residue classes $r$ and $t$ such that $\lambda_{r, t}(p) \sqrt{p} = T$ as a function of $p$ and $T$ is known exactly and involves the Hurwitz class number $H(4p - T^2)$ \cite{Birch}.

Thus, to understand the distribution of $a_p$'s amongst all rank $2$ curves it suffices to understand the frequency of occurences of rank $2$ curves as a function of the residue class $r$, $t$.  
Precisely, we consider the following ratio
\begin{equation}
R_q(X) = \left( \sum_{\substack{E \in \mathcal{F}_{r,t}^+(X) \\ L(1/2, E) = 0}} 1 \right) \biggl\slash
\left( \sum_{\substack{E \in \mathcal{F}_{r',t'}^+(X) \\ L(1/2, E) = 0}} 1 \right).
\end{equation}
Technically, $R_q$ counts curves of even positive rank but it is expected that the number of rank $4$ and higher curves is of a lower order of magnitude than the number of rank $2$ curves (and the numerical evidence supports this!).  

We make the following
\begin{myconj} 
\label{conj:rank2}
Let $R_q = \lim_{X \rightarrow \infty} R_q(X)$.  Then
\begin{align}
\label{eq:rank2}
R_q &= \prod_{\substack{p|q \\ }} \left(\frac{1 - \frac{\lambda_{r,t}(p)}{p^{1/2}} + \frac{1}{p}}{1 - \frac{\lambda_{r',t'}(p)}{p^{1/2}} + \frac{1}{p}} \right)^{1/2}
\\
\nonumber
&= \prod_{\substack{p|q \\ }} \left( \frac{N_p(r,t)}{N_p(r', t')} \right)^{1/2},
\end{align}
where $N_p(r,t)$ is the number of points on the elliptic curve $E_{r,t}$ over $\mf_p$.
\end{myconj}

This formula is similar to that given in Conjecture $2$ of \cite{CKRS}.  A new feature of the above formula is that $\lambda_{r,t}(p)$ and $\lambda_{r', t'}(p)$ can attain any integer values between $-2\sqrt{p}$ and $2 \sqrt{p}$.  

An analogous conjecture may be easily formulated for the relative frequency of rank $3$ curves in arithmetic progressions in $\mathcal{F}'$ with minor modifications: if $R_q(X)$ is given as above but with $L'(\thalf, E)$ replacing $L(\thalf, E)$ then we predict that \eqref{eq:rank2} holds with $N_p(r,t)$ defined to be the number of points on $E_{r,t^2}$.

Conjecture \ref{conj:rank2} is derived by using random matrix theory to deduce information on the distribution of values of $L(\thalf, E)$ from knowledge of the moments of the central values (an idea due to Keating and Snaith \cite{KS}).  

\subsection{Organization of the paper}
In Section \ref{section:background} we recall some necessary material on elliptic curves and $L$-functions.  We derive the shifted moment conjectures in Section \ref{section:deriving}, modulo the precise form of the arithmetical factors, which are calculated in Sections \ref{section:arithmeticalfactor} and \ref{section:ak'}.  We briefly derive Conjecture \ref{conj:rank2} in Section \ref{section:rank2derivation}, and explain the connection with the Riemann Hypothesis in Section \ref{section:RH}.

\subsection{Acknowledgements}
I thank Brian Conrey and David Farmer for suggesting this line of research and for many helpful discussions.  I also thank Duc Khiem Huynh and Nina Snaith for useful feedback.

\section{Background 
and basic properties of the families}
\label{section:background}
In this section we summarize some of the relevant background material on elliptic curves.  Our intended audience contains both random matrix theorists and number theorists who are not specialists in elliptic curves, so we have attempted to provide sufficient details and references.

\subsection{Invariants}
We first describe some of the algebraic invariants associated to an elliptic curve.  Silverman's book \cite{Silverman} is a standard reference.

The general Weierstrass equation of an elliptic curve takes the form
\begin{equation}
\label{eq:Wgeneral}
y^2 + a_1 xy + a_3 y = x^3 + a_2 x^2 + a_4 x + a_6,
\end{equation}
where $a_i \in \mz$.
We are primarily concerned with elliptic curves over $\mq$, but understanding the $L$-function associated to such an elliptic curve involves studying the curve over $\mf_p$ (i.e. reducing the coefficients modulo $p$) for all primes $p$.  Completing the square via $y \rightarrow \thalf (y - a_1 x - a_3)$ gives
\begin{equation}
y^2 = 4x^3 + b_2 x^2 + 2b_4 x + b_6,
\end{equation}
where
\begin{align}
b_2 = a_1^2 + 4a_2, \qquad 
b_4 = 2a_4 + a_1 a_3, \qquad
b_6 = a_3^2 + 4a_6.
\end{align}
The change of variables $x \rightarrow x- \frac{b_2}{12}$ gets rid of the quadratic factor, and then scaling by $x \rightarrow x/36$, $y \rightarrow y/108$ gives
\begin{equation}
y^2 = x^3 - 27c_4 x - 54c_6,
\end{equation}
where
\begin{align}
c_4 = b_2^2 - 24 b_4,\qquad
c_6 = -b_2^3 + 36b_2 b_4 - 216b_6.
\end{align}
These changes of variable are well-defined provided the characteristic of the field is not $2$ or $3$.  Clearly the Weierstrass equation for an elliptic curve is not unique. Table 1.2 of \cite{Silverman} records the effect of the admissible change of variables $x = u^2 x' + r$, $y = u^3 y' + u^2 s x' + t$ on the various quantities $\{a_i\}$, $\{b_i\}$, $\{c_i\}$, and $\Delta$.  We record
\begin{align}
\label{eq:c4'} u^4 c_4'  = c_4,\qquad
u^6 c_6'  = c_6,\qquad u^{12} \Delta'  = \Delta.
\end{align}
Notice that no two curves in $\mathcal{F}$ are isomorphic over $\mq$ because the curves $E_{a,b}$ and $E_{a', b'}$ are isomorphic if and only if there exists $d \in \mq$ such that $a' = d^4 a$ and $b' = d^6 b$. 

In order to choose a `good' Weierstrass equation, we want one that, when reduced modulo various primes, has as good reduction as possible.  To make this precise, we say that the Weierstrass equation \eqref{eq:Wgeneral} is {\em minimal} at $p$ if the largest power of $p$ dividing $\Delta$ cannot be reduced by an admissible change of variables.
Furthermore, we say that \eqref{eq:Wgeneral} is a global minimal Weierstrass equation if it is minimal for all primes $p$.  Chapter 10 of \cite{Knapp} is a good reference for a down-to-earth discussion on global minimal Weierstrass equations.  We quote the following result of N\'{e}ron that appears as Theorem 10.3 of \cite{Knapp}:
\begin{mytheo}[N\'{e}ron]
If $E$ is an elliptic curve over $\mq$, then there exists an admissible change of variables over $\mq$ such that the resulting equation is a global minimal Weierstrass equation.  Two such resulting global minimal Weierstrass equations are related by an admissible change of variables with $u = \pm 1$ and with $r,s,t \in \mz$.
\end{mytheo}
Now we claim that the Weierstrass equation \eqref{eq:weierstrass} for each $E_{a,b} \in \mathcal{F}$ is a global minimal equation.  We use the condition that if $p^{12} \nmid \Delta$ or $p^4 \nmid c_4$ or $p^6 \nmid c_6$, then the Weierstrass equation is minimal at $p$, where here $c_4 = - 2^4 3 a$, $c_6 = - 2^5 3^3 b$.  This condition is Lemma 10.1 of \cite{Knapp} and is essentially Remark 1.1 of Section VII.1 of \cite{Silverman}, but is easily seen from inspection of \eqref{eq:c4'}. It is immediate from this test that \eqref{eq:weierstrass} is minimal at all $p > 3$, using the condition that if $p^4 | a$ then $p^6 \nmid b$.  We have $2^4 || \Delta$ and $3 \nmid \Delta$ since $(4r^3 + 27t^2, 6) = 1$, so the equation is minimal at $p \leq 3$.  A point to take from this discussion is that it is easy to specify light conditions that ensure minimality of a given Weierstrass equation.  

\subsection{The $L$-function}
The {\em conductor} $N$ associated to $E$ is a certain divisor of the minimal discriminant (which is the discriminant of the global minimal Weierstrass equation).  The conductor and the minimal discriminant have the same prime factors, and for $p > 3$ we have $p || N$ if $E$ has a node (a double root) modulo $p$, and $p^2 || N$ if $E$ has a cusp (a triple root) modulo $p$.  For $p \leq 3$ it is not so simple to give a characterization of the power of $p$ dividing $N$, but it can be found using Tate's algorithm, which is described in \cite{Silverman2}, pp. 363-368.

Given a global minimal Weierstrass equation of the form \eqref{eq:weierstrass}, the $L$-function attached to $E_{a,b}$ is given by
\begin{equation}
\label{eq:Ldef}
L(s, E_{a,b}) = \sum_{n=1}^{\infty} \frac{\lambda_{a,b}(n)}{n^s} = \prod_p \left( 1 - \frac{\lambda_{a,b}(p)}{p^s} + \frac{\psi_N(p)}{p^{2s}} \right)^{-1},
\end{equation}
where for $p \neq 2$, 
\begin{equation}
\lambda_{a, b}(p) = \frac{1}{\sqrt{p}} (p + 1 - \#E(\mf_p)) = - \frac{1}{\sqrt{p}} \sum_{x \shortmod{p}} \left(\frac{x^3 + ax + b}{p} \right),
\end{equation}
and $\psi_N$ is the principal Dirichlet character of modulus the conductor $N$ of $E$.  If $p=2$ then \eqref{eq:weierstrass} has a cusp and $\lambda_{a,b}(2^k) = 0$ for all $k$.  The sum and product converge absolutely provided $\text{Re}(s) > 1$, using Hasse's bound $|\lambda_{a,b}(n)| \leq d(n)$.
The famous modularity theorem \cite{W}, \cite{TW}, \cite{BCDT} shows that the completed $L$-function
\begin{equation}
\Lambda(s, E) = \left(\frac{\sqrt{N}}{2 \pi} \right)^{s + \half} \Gamma(s + \thalf) L(s, E),
\end{equation}
is entire and satisfies the functional equation $\Lambda(s, E) = w_E \Lambda(1-s, E)$ where $w_E = \pm 1$ is called the root number.  We set
\begin{equation}
\label{eq:X}
X_E(s) =  \frac{\Gamma(\frac32 - s)}{\Gamma(\onehalf + s)} \left(\frac{\sqrt{N}}{2 \pi} \right)^{1 - 2s},
\end{equation}
so that
\begin{equation}
L(s, E) = w_E X_E(s) L(1-s, E).
\end{equation}
Note that $X_E(s)$ only depends on the conductor $N$.  We have normalized the $L$-function to have central point $s=\thalf$.

The root number can be effectively computed; it is given by a product of local root numbers.  We state the following result that appears as Proposition 3.1 of \cite{Y}.
\begin{myprop}
\label{prop:rootnumber}
Suppose $4a^3 + 27b^2$ is squarefree.  Then the root number of $E_{a,b}:y^2 = x^3 + ax + b$ is given by
\begin{equation}
w_{E_{a,b}} = \mu(4a^3 + 27b^2) \left(\frac{a}{3b}\right) \chi_4(b) (-1)^{a+1} \epsilon_2,
\end{equation}
where $(\frac{\cdot}{\cdot})$ is the Jacobi symbol, $\chi_4$ is the primitive Dirichlet character of conductor $4$, and $\epsilon_2$ is the local root number at $p=2$.
\end{myprop}
The local root number at $2$ is difficult to state explicitly because there are many possible cases.  The point is that $\mu(4a^3 + 27b^2)$ is expected to oscillate independently of the other factors, so that the root number is evenly distributed between $\pm 1$.

An exercise with M\"{o}bius inversion shows that $|\mathcal{F}(X)| \sim  \frac{X^{5/6}}{9q^{2}  \zeta_{6q}(10)}$, where $\zeta_{6q}$ is given by the same Euler product as $\zeta$ but with the local factors at $p | 6q$ removed. A proof of this is contained in the proof of Lemma \ref{lem:brain} (take $n_1 = 1$, $k=1$).  Similarly, $|\mathcal{F}'(X)| \sim \frac{X^{7/12}}{9\zeta_6(7)}$.

Each curve $E_{a,b^2} \in \mathcal{F}'$ has the obvious point $(0,b)$.  The Lutz-Nagell criterion easily shows that this point is torsion (has finite order) very rarely.  We paraphrase Corollary 7.2 of \cite{Silverman}:
\begin{mytheo}[Lutz-Nagell]
\label{theorem:L-N}
Let $E$ be an elliptic curve given by \eqref{eq:weierstrass}.  Suppose $(x,y)$ is a non-zero torsion point.  Then $x, y \in \mz$ and either $y=0$ or $y^2 | 4a^3 + 27b^2$.
\end{mytheo}
Thus if $(0,b)$ is a torsion point then $b^2 | 4a^3$.  Clearly the number of $E_{a,b} \in \mathcal{F}'(X)$ such that $b^2 | 4a^3$ is $O(X^{\frac13 + \varepsilon})$ (which should be compared to $X^{7/12}$, the total number of curves).

\subsection{Chebyshev polynomials}
\label{section:Chebyshev}
We take a brief detour to summarize some relevant facts needed about Chebyshev polynomials (of the second kind) $U_k$.  A 
reference for the necessary formulas is Section 8.94 of \cite{GR}.  By definition,
\begin{equation}
U_n(\cos \theta) = \frac{\sin(n+1)\theta}{\sin\theta}.
\end{equation}
These satisfy the recursion formula (8.941.2 of \cite{GR})
\begin{equation}
U_{n+2}(x) - 2x U_{n+1}(x) + U_{n}(x) = 0
\end{equation}
which is equivalent to the formal identity
\begin{equation}
\label{eq:Ukgenseries}
\sum_n U_n(x) t^n = \frac{1}{1 - 2 x t + t^2},
\end{equation}
the sum being absolutely convergent for $|x|, |t| < 1$.
Since the Hecke operators satisfy essentially the same recurrence relation, we have
\begin{equation}
\label{eq:lambdapower}
\lambda_E(p^j) = \begin{cases} U_j\left(\frac{\lambda_E(p)}{2}\right), \qquad &\text{if $(p,N) = 1$},\\ \lambda_E^j(p), \qquad &\text{if $p | N$}. \end{cases}
\end{equation}
The Chebyshev polynomials $U_n(\cos \theta)$ form an orthonormal system with respect to the Sato-Tate measure $\frac{2}{\pi} \sin^2{\theta} d\theta := d\mu_{ST}$, where the integration is over the interval $[0, \pi]$.  

It will be useful to represent a product of Chebyshev polynomials in terms of Chebyshev polynomials.
Let $c_l = c_l(e_1, \ldots, e_k)$ be defined by
\begin{equation}
\label{eq:Ukortho}
U_{e_1}(x) \cdots U_{e_k}(x) = \sum_{l} c_l U_l(x),
\end{equation}
and set $f = e_1 + \ldots + e_k$.  Note that by parity considerations (namely, the parity of $U_k(x)$ as a function of $x$ is the same as the parity of $k$), $c_l = 0$ if $l \not \equiv f \pmod{2}$.  Furthermore, by degree considerations $c_l = 0$ if $l > f$.  The orthogonality relation gives
\begin{align}
\label{eq:cr}
c_l = \int U_l(\cos\theta) \prod_{j=1}^{k} U_{e_j} (\cos\theta) d\mu_{ST}.
\end{align} 

\section{Deriving the conjectures}
\label{section:deriving}
\subsection{Central values of the family of all elliptic curves}
We begin by deriving Conjecture \ref{conj:Pkconj}, however we delay the computation of the exact form of the arithmetical factor until Section \ref{section:arithmeticalfactor}.  Actually, we find a conjectural formula for a product of $k$ $L$-functions at points shifted slightly away from the central point.  Conjecture \ref{conj:Pkconj} is a limiting form of this more general conjecture.  This generality also allows us to compute the central values of the derivatives by differentiation with respect to the shift parameters.

We want the moment
\begin{equation}
\frac{1}{|\mathcal{F}(X)|} \sum_{E \in \mathcal{F}(X)} L(\thalf + \alpha_1,E) \dots L(\thalf + \alpha_k,E),
\end{equation}
which is analogous to 4.1.4 of \cite{CFKRS}.  Actually we will write a conjecture for the more symmetric expression
\begin{equation}
\label{eq:symmetricmoment}
\frac{1}{|\mathcal{F}(X)|} \sum_{E \in \mathcal{F}(X)} \mathcal{Z}(E), 
\end{equation}
where
\begin{equation}
\mathcal{Z}(E) = X^{-\half}_E(\thalf + \alpha_1) \dots X^{-\half}_E(\thalf + \alpha_k )L(\thalf + \alpha_1,E) \dots L(\thalf + \alpha_k,E)
\end{equation}
and where recall $X_E$ satisfies
\begin{equation}
X_E^{-\half}(\thalf + \alpha) L(\thalf + \alpha,E) = w_E X_E^{-\half}(\thalf - \alpha) L(\thalf - \alpha,E).
\end{equation}
For now we just work with the $L$-functions.

For each $L$ function we write a kind of
approximate functional equation as follows
\begin{equation}
L(\thalf + \alpha, E) = \sum_n \frac{\lambda_E(n)}{n^{\half + \alpha}} + w_E X_E(\thalf + \alpha) \sum_n \frac{\lambda_E(n)}{n^{\half - \alpha}}.
\end{equation}
Consider the term obtained by taking the `first part' of each approximate functional equation:
\begin{equation}
\frac{1}{|\mathcal{F}(X)|} \sum_{E \in \mathcal{F}(X)} \sum_{n_1, \dots, n_k} \frac{\lambda_E(n_1) \dots \lambda_E(n_k)}{n_1^{\half + \alpha_1} \dots n_k^{\half + \alpha_k}}.
\end{equation}
Now replace each summand and replace it with its average, say
\begin{equation}
\label{eq:pickup}
\sum_{n_1, \dots, n_k} \frac{R_{r,t}(n_1,\dots,n_k)}{n_1^{\half + \alpha_1} \dots n_k^{\half + \alpha_k}},
\end{equation}
where
\begin{equation}
R_{r,t}(n_1,\dots, n_k) = \lim_{X \rightarrow \infty} \frac{1}{|\mathcal{F}_{r,t}(X)|} \sum_{E \in \mathcal{F}_{r,t}(X)} \lambda_E(n_1) \dots \lambda_E(n_k).
\end{equation}
 
We now derive the necessary expected value.
To this end, we make the following
\begin{mydefi}
Let $n_1, \ldots, n_k$ be positive integers, set $n = [n_1, \ldots, n_k]$ (the least common multiple), and let $n^*$ be the product of primes dividing $n$.  Define $Q^*$ by
\begin{equation}
Q^*(n_1, \ldots, n_k) = \frac{1}{{n^*}^2} \sum_{a, b \shortmod{n^*}} \lambda_{a, b}(n_1) \cdots \lambda_{a, b}(n_k).
\end{equation}
Furthermore, set $n_i = m_i l_i$, where $(m_i, 6q) = 1$ and every prime dividing $l_i$ also divides $6q$.  Then set
\begin{equation}
Q^*_{r,t}(n_1, \ldots, n_k) = \lambda_{r,t}(l_1) \ldots \lambda_{r,t}(l_k) Q^*(m_1, \ldots, m_k) \prod_{\substack{p | n \\ p \notdiv 6q}} (1 - p^{-10})^{-1}.
\end{equation}
\end{mydefi}

The desired expected value is given by the following
\begin{mylemma}  We have
\label{lem:brain}
\begin{equation}
R_{r,t}(n_1,\ldots, n_k) = Q^*_{r,t}(n_1, \ldots, n_k).
\end{equation}
Furthermore, $Q^*(n_1,\ldots,n_k)$ is multiplicative.  That is, if $n_i = n_i' n_i''$ for $1\leq i \leq k$ with $(n_1' \dots n_k', n_1'' \dots n_k'') = 1$ then
\begin{equation}
Q^*(n_1,\ldots,n_k) = Q^*(n_1', \ldots, n_k') Q^*(n_1'', \ldots, n_k'').
\end{equation}
\end{mylemma}
\begin{proof}
Recall
\begin{equation}
R_{r,t}(n_1,\ldots, n_k) = \lim \frac{1}{|\mathcal{F}(X)|} \mathop{\sum \sum}_{\substack{|a| \leq X^{1/3}, |b| \leq  X^{1/2} \\ p^4 | a \Rightarrow p^6 \nmid b \\ a \equiv r \shortmod{6q} \\ b \equiv  t \shortmod{6q}}}
\lambda_{a, b}(n_1) \cdots \lambda_{a, b}(n_k),
\end{equation}
where we originally defined $\lambda_{a,b}(n)$ to be the $n$-th coefficient in the Dirichlet series expansion of $L(s,E_{a,b})$.  In case \eqref{eq:weierstrass} is a global minimal Weierstrass equation for $E_{a,b}$ (a condition that is met for all curves in $\mathcal{F}$), then 
\begin{equation}
\label{eq:lambda2}
\lambda_{a,b}(p) = -\frac{1}{\sqrt{p}} \sum_{x \shortmod{p}} \left(\frac{x^3 + ax + b}{p} \right).
\end{equation}
We extend the definition of $\lambda_{a,b}(p)$ to arbitrary integers $a$ and $b$ and primes $p$ by using \eqref{eq:lambda2} for $p > 2$, and setting $\lambda_{a,b}(2) = 0$.  We then extend to prime powers using \eqref{eq:lambdapower}, replacing the condition $(p,N) = 1$ with $(p, 16(4a^3 + 27b^2)) = 1$, and we extend to composite integers by multiplicativity.  This formulation will allow us to easily sum over $a$ and $b$ where \eqref{eq:weierstrass} is not necessarily a global minimal Weierstrass equation,

A slight generalization of the M\"{o}bius inversion formula gives that
\begin{equation}
\sum_{\substack{d^4 | a \\ d^6 | b}} \mu(d) = \begin{cases} 1, \qquad \text{if there does not exist a $p$ such that $p^4 |a$ and $p^6 | b$}, \\0, \qquad \text{otherwise}. \end{cases}
\end{equation}
Thus
\begin{align}
R_{r,t}(n_1,\ldots, n_k) = \lim \frac{1}{|\mathcal{F}(X)|} \sum_{\substack{d \leq X^{\frac{1}{12}} \\(d, 6q) = 1}} \mu(d) \mathop{\sum \sum}_{\substack{|a| \leq d^{-4} X^{1/3} \\ |b| \leq d^{-6} X^{1/2} \\ a \equiv \overline{d}^{4} r \shortmod{6q} \\ b \equiv \overline{d}^6 t \shortmod{6q}}}
\lambda_{ad^4, bd^6}(n_1) \cdots \lambda_{ad^4, bd^6}(n_k).
\end{align}
The condition $(d,6q) = 1$ follows from the fact that $(4r^3 + 27t^2, 6q) = 1$.
Now we claim that $\lambda_{ad^4, bd^6}(m) = \lambda_{a, b}(m)$ when $(m, d) = 1$ and $\lambda_{ad^4, bd^6}(m) = 0$ when $(m, d) > 1$.  By multipicativity and the fact that prime powers are determined by primes, it suffices to check for $m$ prime.  Notice that in case $p | (a,b)$ then the sum \eqref{eq:lambda2} vanishes.  Thus if $p | d$ then $\lambda_{ad^4, bd^6}(p) = 0$.  On the other hand, if $(p, d) = 1$ then the change of variables $x \rightarrow d^2 x$ modulo $p$ gives
\begin{equation}
\sum_{x \shortmod{p}} \left(\frac{x^3 + ad^4x + bd^6}{p} \right) =  \left(\frac{d^6}{p}\right) \sum_{x \shortmod{p}} \left(\frac{x^3 + ax + b}{p} \right),
\end{equation}
so $\lambda_{ad^4, bd^6}(p) = \lambda_{a,b}(p)$ as claimed.  Thus for $(d,n) =1$ we have
\begin{equation}
\lambda_{ad^4, bd^6}(n_1) \cdots \lambda_{a, b}(n_k) = \lambda_{r, t}(l_1) \cdots \lambda_{r, t}(l_k) \lambda_{a, b}(m_1) \cdots \lambda_{a, b}(m_k),
\end{equation}
by multiplicativity and since $\lambda_{ad^4, bd^6}(l_i) = \lambda_{r,t}(l_i)$.

Hence we have
\begin{multline}
R_{r,t}(n_1,\ldots, n_k) 
\\
= \lim \frac{1}{|\mathcal{F}(X)|} \lambda_{r,t}(l_1) \ldots \lambda_{r,t}(l_k) \sum_{\substack{d \leq X^{\frac{1}{12}} \\(d, 6nq) = 1}} \mu(d) \mathop{\sum \sum}_{\substack{|a| \leq d^{-4} X^{1/3} \\ |b| \leq d^{-6} X^{1/2} \\ a \equiv \overline{d}^{4} r \shortmod{6q} \\ b \equiv \overline{d}^6 t \shortmod{6q}}}
\lambda_{a, b}(m_1) \cdots \lambda_{a, b}(m_k).
\end{multline}
Now $\lambda_{a,b}(m_1) \dots \lambda_{a,b}(m_k)$ is periodic in $a$ and $b$ with period equal to the product of primes dividing the least common multiple of $m_1, \dots, m_k$, say $m^*$.  Breaking up the sum over $a$ and $b$ into arithmetic progressions modulo $6qm^*$ gives
\begin{multline}
R_{r,t}(n_1,\ldots, n_k) = \lim \frac{1}{|\mathcal{F}(X)|} \lambda_{r,t}(l_1) \ldots \lambda_{r,t}(l_k) \frac{4X^{5/6}}{36 q^2} \sum_{\substack{d \leq X^{\frac{1}{12}} \\(d, 6mq) = 1}} \frac{\mu(d)}{d^{10}} 
\\
\frac{1}{{m^*}^2} \sum_{\substack{\alpha \shortmod{m^*} \\ \beta \shortmod{m^*}}} 
\lambda_{\alpha, \beta}(m_1) \cdots \lambda_{\alpha, \beta}(m_k),
\end{multline}  
which simplifies to
\begin{equation}
R_{r,t}(n_1,\ldots, n_k) =\lim \frac{1}{|\mathcal{F}(X)|} \lambda_{r,t}(l_1) \ldots \lambda_{r,t}(l_k) \frac{X^{5/6}}{9 q^2 \zeta_{6mq}(10)} Q^*(m_1,\dots,m_k).
\end{equation} 
Taking $k=1$, $n_1=1$ gives
\begin{equation}
|\mathcal{F}(X)| \sim \frac{X^{5/6}}{9 q^2 \zeta_{6q}(10)},
\end{equation}
so
\begin{equation}
R_{r,t}(n_1,\ldots, n_k) = \lambda_{r,t}(l_1) \ldots \lambda_{r,t}(l_k) Q^*(m_1,\dots,m_k) \frac{\zeta_{6q}(10)}{\zeta_{6mq}(10)}.
\end{equation}
Using that
\begin{equation}
\frac{\zeta_{6q}(10)}{\zeta_{6mq}(10)} = \prod_{\substack{p | m \\ p \nmid 6q}} (1- p^{-10})^{-1} =\prod_{\substack{p | n \\ p \nmid 6q}} (1- p^{-10})^{-1},
\end{equation}
completes the proof that $R_{r,t} = Q^*_{r,t}$

Now we show that $Q^*(n_1,\ldots, n_k)$ is multiplicative.  To simplify notation only, we take $k=1$ and show $Q^*(mn) = Q^*(m) Q^*(n)$ provided $(m,n) = 1$, where $m$ and $n$ are squarefree.  Extending to the general case is straightforward.  By the Chinese remainder theorem we may write all representatives $a \pmod{mn}$ uniquely in the form $a_1 m \overline{m} + a_2 n \overline{n}$, where $a_1$ runs modulo $n$, $a_2$ runs modulo $m$, $m \overline{m} \equiv 1 \pmod{n}$, and $n \overline{n} \equiv 1 \pmod{m}$, and similarly for $b$.  Then we get, using $\lambda(mn) = \lambda(m) \lambda(n)$,
\begin{align}
Q^*(mn) & = \frac{1}{m^2 n^2} \sum_{a_1, b_1 \shortmod{n}} \sum_{a_2, b_2 \shortmod{m}} \lambda_{a_1 m \overline{m} + a_2 n \overline{n}, b_1 m \overline{m} + b_2 n \overline{n}}(mn) 
\\
& = \frac{1}{m^2 n^2} \sum_{a_1, b_1 \shortmod{n}} \sum_{a_2, b_2 \shortmod{m}} \lambda_{a_1, b_1}(n) \lambda_{a_2,b_2}(m) = Q^*(m) Q^*(n). \qedhere
\end{align}
\end{proof}

Now we continue with our derivation of the moment conjecture, picking up with \eqref{eq:pickup}.
We extend the summation over the $n_i$ to all positive integers and set
\begin{equation}
N_E(\alpha_1, \ldots, \alpha_k) 
 =
X_E^{-\half}(\thalf + \alpha_1) \cdots X_E^{-\half}(\thalf + \alpha_k)
\sum_{n_1} \cdots \sum_{n_k} \frac{Q_{r,t}^*(n_1, \ldots, n_k)}{n_1^{\half + \alpha_1} \cdots n_k^{\half + \alpha_k }}.
\end{equation}
We will obtain a meromorphic continuation of $N_E(\alpha_1, \ldots, \alpha_k)$ to a neighborhood of $(0,\ldots,0)$.
Let
\begin{equation}
M_E(\alpha_1, \ldots, \alpha_k) = \sum_{\substack{\epsilon_1, \dots, \epsilon_k \in \{-1,1\} \\ \epsilon_1 \dots \epsilon_k = 1}} N_E(\epsilon_1 \alpha_1, \ldots, \epsilon_k \alpha_k).
\end{equation}
Here $M_E$ is obtained by summing $N_E$ over all possible ways of swapping an even number of $\alpha_i$'s with their negatives.  The moment \eqref{eq:symmetricmoment} is clearly invariant under these symmetries, so in effect this is the simplest way to manipulate $N_E(\alpha_1, \ldots, \alpha_k)$ to obtain an expression with these symmetries.  This procedure is different than that stated in \cite{CFKRS}, but is essentially equivalent.

The general conjecture is then (\cite{CFKRS}, 4.1.6)
\begin{equation}
\sum_{E \in \mathcal{F}(X)} \mathcal{Z}(E)
=
\sum_{E \in \mathcal{F}(X)} M_E(\alpha_1, \ldots, \alpha_k) \left(1 + O(N_E^{-\delta}) \right) 
\end{equation}
for some $\delta > 0$.

The next step is to express this answer in a more usable form.  
We write
\begin{multline}
M_E(\alpha_1, \ldots, \alpha_k) 
= \\
\mathop{\sum \cdots \sum}_{\epsilon_1 \cdots  \epsilon_k =1} 
X_E^{-\half}(\thalf + \epsilon_1 \alpha_1) \cdots X_E^{-\half}(\thalf + \epsilon_k \alpha_k) 
H(\epsilon_1 \alpha_1, \ldots, \epsilon_k \alpha_k),
\end{multline}
where
\begin{equation}
\label{eq:Hdef1}
H(z_1, \ldots, z_k) =  \sum_{n_1} \cdots \sum_{n_k} \frac{Q_{r,t}^*(n_1, \ldots, n_k)}{n_1^{\half + z_1} \cdots n_k^{\half + z_k}}.
\end{equation}

Using the multiplicativity of $Q^*$ we write
\begin{align}
H(z_1, \ldots, z_k) =
\prod_p \sum_{e_1} \cdots \sum_{e_k} \frac{Q^*_{r,t}(p^{e_1}, \ldots, p^{e_k})}{p^{e_1(\half + z_1) + \ldots + e_k(\half + z_k)}} &
 \\
\label{eq:Hdef2}
 = 
\left( \prod_{p |6q} \sum_{e_1} \cdots \sum_{e_k}  \frac{\lambda_{r,t}(p^{e_1})\dots \lambda_{r,t}(p^{e_k}) }{p^{e_1(\half + z_1) + \ldots + e_k(\half + z_k)}} \right)
&
\left( \prod_{p \nmid 6q} \sum_{e_1} \cdots \sum_{e_k} \frac{\delta(p) Q^*(p^{e_1}, \ldots, p^{e_k})}{p^{e_1(\half + z_1) + \ldots + e_k(\half + z_k)}} \right),
\end{align}
where $\delta(p) = (1- p^{-10})^{-1}$ if $e_1 + \ldots + e_k > 0$ and $\delta(p) = 1$ otherwise.  Note that for $p | 3q$
\begin{equation}
\sum_{e=0}^{\infty} \frac{\lambda_{r,t}(p^{e})}{p^{e(\half + z)}} = \left(1 - \frac{\lambda_{r,t}(p)}{p^{\half + z}} + \frac{1}{p^{1 + 2z}} \right)^{-1},
\end{equation}
using \eqref{eq:Ldef} and the fact that $(6q, 4r^3 + 27t^2) = 1$.

We wish to determine the polar behavior of $H(z_1,\dots,z_k)$ for $z_i$ near $0$.  Since $Q^*(m_1,\dots,m_k) \ll (m_1 \dots m_k)^{\varepsilon}$, it suffices to consider the contribution from $e_1 + \dots + e_k \leq 2$.  We compute $H$ precisely in Section \ref{section:arithmeticalfactor} with Proposition \ref{prop:Q}.  For now we simply state that $Q^*(p^j,1,\dots,1) =0$ for $j=1,2$ (actually it vanishes for $j < 10$) and that $Q^*(p,p, 1, \ldots 1) = 1 - p^{-1}$ (see Corollary \ref{coro:Q*approx} for this identity).  By factoring out the appropriate zeta functions we have
\begin{equation}
H(z_1, \ldots, z_k) = \left(\prod_{1 \leq i < j \leq k} \zeta(1 + z_i + z_j) \right) A_k(z_1, \ldots, z_k),
\end{equation}
where $A_k$ is the arithmetical factor which is holomorphic and nonzero in a neighborhood of $0$.

The next step in the recipe is to use Lemma 2.5.2 of \cite{CFKRS} to express the permutation sum in terms of a multiple contour integral.  To this end, set
\begin{equation}
f(s) = \zeta(1 + s),
\end{equation}
\begin{equation}
F(a_1, \ldots, a_k) = A_k( a_1, \ldots, a_k) X_E^{-\half}(\thalf + a_1) \cdots X_E^{-\half}(\thalf + a_k),
\end{equation}
and
\begin{equation}
K(a_1, \ldots, a_k) = F(a_1, \ldots, a_k) \prod_{1 \leq i < j \leq k} f(a_i + a_j).
\end{equation}
Then $f$, $F$, and $K$ satisfy the conditions of Lemma 2.5.2, namely $F$ is a symmetric function, holomorphic near the origin and $f$ has a simple pole of residue $1$ at $s=0$ but is otherwise holomorphic near $s=0$.  With this setup we see that
\begin{equation}
M_E(\alpha_1, \ldots, \alpha_k) = \sum_{\epsilon_j = \pm 1} \thalf (1 + \prod_{j = 1}^k \epsilon_j ) K(\epsilon_1 \alpha_1, \ldots, \epsilon_k \alpha_k).
\end{equation}
So by Lemma 2.5.2,
\begin{multline}
\label{eq:shiftedmoment}
M_E(\alpha_1, \ldots, \alpha_k) = \\
\frac{(-1)^{k(k-1)/2}}{(2\pi i)^k} \frac{2^{k-1}}{k!} \oint \cdots \oint K(z_1, \ldots, z_k) 
\frac{\Delta(z_1^2, \ldots, z_k^2)^2(\prod_{j=1}^k z_j + \prod_{j=1}^k \alpha_j)}{\prod_{i=1}^k \prod_{j=1}^k (z_i - \alpha_j)(z_i + \alpha_j)} 
dz_1 \cdots dz_k,
\end{multline}
where $\Delta$ is the Vandermonde determinant 
\begin{equation}
\Delta(z_1, \ldots, z_k) = \prod_{1 \leq i < j \leq k} (z_j - z_i),
\end{equation}
and where the paths of integrations enclose the $\pm \alpha_j$'s.  Taking the $\alpha_i$'s to be zero we obtain Corollary \ref{conj:Pkconj}.

The precise form of the arithmetical factor is given in Proposition \ref{prop:Ak}.  The computation of the arithmetical factor relies on further knowledge of $Q^*$, which is discussed in Section \ref{section:arithmeticalfactor}.

\subsection{The positive rank family}
We now turn to the derivation of Conjectures \ref{conj:momentfirstterm'} and \ref{conj:Qkconj}.  We follow the same procedure as the previous section to obtain a conjecture on the shifted product
\begin{equation}
\frac{1}{|\mathcal{F}'(X)|} \sum_{E \in \mathcal{F}'(X)} X_E^{-\half}(\thalf + \alpha_1) \dots X_E^{-\half}(\thalf + \alpha_k) L(\thalf + \alpha_1, E) \cdots L(\thalf + \alpha_k, E).
\end{equation}
In this case since the central values almost always vanish this quantity will average to zero, at least if one of the shift parameters is zero.  However, we can differentiate the moment conjecture \eqref{eq:shiftedmoment} with respect to each $\alpha_i$ and set $\alpha_i = 0$ to obtain a conjecture for the moments of the first derivatives of the $L$-functions at the central point.  The derivation  of the shifted moment conjecture goes through essentially unchanged.

Before stating the conjecture we first set some new notation for this family.  Set
\begin{equation}
Q_{\square}^*(n_1, \ldots, n_k) = \frac{1}{{n^*}^2} \sum_{a, b \shortmod{n^*}} \lambda_{a,b^2}(n_1) \cdots \lambda_{a,b^2}(n_k),
\end{equation}
and
\begin{equation}
Q'(n_1, \ldots, n_k) = Q_{\square}^*(n_1, \ldots, n_k) \prod_{p | n} (1 - p^{-7})^{-1}.
\end{equation}
The Dirichlet series formed from $Q'$ is given by
\begin{equation}
H'(z_1, \ldots, z_k) = \prod_p \sum_{e_1} \cdots \sum_{e_k} \frac{Q'(p^{e_1}, \ldots, p^{e_k})}{p^{e_1(\half + z_1) + \ldots + e_k(\half + z_k)}}.
\end{equation}
As shorthand we write
\begin{equation}
Y_E(z_1, \ldots, z_k) =X_E^{-\half}(\thalf + z_1) \cdots X_E^{-\half}(\thalf + z_k).
\end{equation}

The general shifted moment conjecture then reads
\begin{equation}
\sum_{E \in \mathcal{F}'(X)} \mathcal{Z}(E) = \sum_{E \in \mathcal{F}'(X)} M_E(\alpha_1, \ldots, \alpha_k)\left(1 + O(N_E^{-\delta}) \right),
\end{equation}
where
\begin{multline}
\label{eq:rankoneM}
M_E(\alpha_1, \ldots, \alpha_k) = \\
c_k \oint \cdots \oint H'(z_1, \ldots, z_k) 
\frac{\Delta(z_1^2, \ldots, z_k^2)^2(\prod_{j=1}^k z_j + \prod_{j=1}^k \alpha_j)}{\prod_{i=1}^k \prod_{j=1}^k (z_i - \alpha_j)(z_i + \alpha_j)} 
Y_E(z_1, \ldots, z_k) dz_1 \cdots dz_k,
\end{multline}
and where $c_k = \frac{(-1)^{k(k-1)/2}}{(2\pi i)^k} \frac{2^{k-1}}{k!}$.

The behavior of $M_E$ for the positive rank family is drastically different from the family of all elliptic curves because the polar behavior of $H'$ is different than that of $H$.  A thorough study of $Q'$ is undertaken in Section \ref{section:ak'}.  For now we use that $Q'(p, 1, \ldots, 1) = -p^{-1/2} + O(p^{-3/2})$, $Q'(p, p, 1, \ldots, 1) = 1 + O(p^{-1})$, and $Q'(p^2, 1, \ldots, 1) = 0$ to deduce the polar behavior of $H'$.  Precisely,
\begin{equation}
\label{eq:H'}
H'(z_1, \ldots, z_k) = \frac{\prod_{1 \leq i < j \leq k} \zeta(1 + z_i + z_j)}{\prod_{1 \leq i \leq k} \zeta(1 + z_i)} A_k'(z_1, \ldots, z_k)
\end{equation} 
where $A_k'$ is given by an absolutely and uniformly convergent Euler product in a neighborhood of $(0, \ldots, 0)$.

We now check that the moment conjecture is consistent with the Birch and Swinnerton-Dyer conjecture, that is that $M_E(\alpha_1,\dots,\alpha_k) = 0$ if some $\alpha_i = 0$.  

\begin{myprop}
The function $M_E(\alpha_1,\dots,\alpha_k)$ given by \eqref{eq:rankoneM} vanishes at $\alpha_i = 0$, for any $i$.
\end{myprop}
\begin{proof}
We go back to the original representation of $M$ as a permutation sum of the form
\begin{equation}
S = \sum_{\epsilon_1, \dots, \epsilon_k \in \{-1, 1\}} \thalf (1 + \prod_{1\leq l \leq  k} \epsilon_l) f(\epsilon_1 \alpha_1, \dots, \epsilon_k \alpha_k) \prod_{1 \leq i < j \leq k} \zeta(1 + \epsilon_i \alpha_i + \epsilon_j \alpha_j),
\end{equation}
where $f$ is a symmetric, regular function near the origin.  We know from Lemma 2.5.2 of \cite{CFKRS} that these conditions ensure that such a sum is holomorphic in terms of the shift parameters near the origin.  In the case of $M$ given by \eqref{eq:rankoneM}, we furthermore know that $(z_1 \dots z_k)^{-1} f(z_1,\dots, z_k) = g(z_1, \dots, z_k)$, say, is regular at the origin.  Thus
\begin{equation}
S = \alpha_1 \dots \alpha_k \sum_{\epsilon_1, \dots, \epsilon_k \in \{-1, 1\}} \thalf (1 + \prod_{1\leq l \leq  k} \epsilon_l) g(\epsilon_1 \alpha_1, \dots, \epsilon_k \alpha_k) \prod_{1 \leq i < j \leq k} \zeta(1 + \epsilon_i \alpha_i + \epsilon_j \alpha_j),
\end{equation}
so $(\alpha_1 \dots \alpha_k)^{-1} S$ is regular at the origin, and hence $S$ vanishes when any shift parameter is set to zero.
\end{proof}

Let $M_E'$ be the derivative of $M_E(\alpha_1, \ldots, \alpha_k)$ with respect to all $\alpha_i$ evaluated at $\alpha_1 = \ldots = \alpha_k = 0$.  Now we derive Conjecture \ref{conj:Qkconj} by computing $M_E'$ using \eqref{eq:rankoneM}.  Again we consider the terms with $\prod z_j$ and $\prod \alpha_j$ separately.  It is obvious that the former term is even with respect to each $\alpha_i$ so that differentiating at $\alpha_i = 0$ yields zero.  By parity considerations it follows quickly that
\begin{equation}
\left.\frac{\partial}{\partial \alpha_1} \cdots \frac{\partial}{\partial \alpha_k} \left(\prod_{1 \leq j \leq k} \alpha_j \prod_{i=1}^k \prod_{j=1}^k (z_i^2 - \alpha_j^2)^{-1}\right)\right|_{\alpha_1 = \dots = \alpha_k = 0} = \prod_{i=1}^k z_i^{-2k}.
\end{equation} 
We now see that
\begin{multline}
M_E' = 
\\
c_k' \oint \cdots \oint A_k'(z_1, \ldots, z_k) 
\prod_{1 \leq i < j \leq k} \zeta(1 + z_i + z_j)\frac{\Delta(z_1^2, \ldots, z_k^2)^2 }{\prod_{i=1}^k  z_i^{2k} \zeta(1 + z_i)} 
Y_E(z_1, \ldots, z_k) dz_1 \cdots dz_k.
\end{multline}
By comparison with \eqref{eq:PkN} we see that $M_E'$ has the same order of magnitude as $P_k(N)$.  This completes the derivation of Conjecture \ref{conj:Qkconj}.

\section{The arithmetical factor for the family of all elliptic curves}
\label{section:arithmeticalfactor}
In order to understand the arithmetical factor it is necessary to understand the behavior of $Q^*(n_1, \ldots, n_k)$.  Because of the multiplicativity of $Q^*$ it suffices to understand
\begin{equation}
\label{eq:Q*}
Q^*(p^{e_1}, \ldots, p^{e_k}) = \frac{1}{p^2} \sum_{a, b \shortmod{p}} \lambda_{a, b}(p^{e_1}) \cdots \lambda_{a, b}(p^{e_k}).
\end{equation}
We desire a usable formula for this expression.  Such a derivation is the purpose of this section.
\subsection{The case $k=1$}
We first derive a formula for $Q^*(p^j)$.  To this end, we state
\begin{myprop} 
\label{prop:Q}
Set 
\begin{equation}
\label{eq:Q}
Q(p^j) =  \sum_{a, b \shortmod{p}} p^{j/2} \lambda_{a,b}(p^j)
\end{equation}
and let $Tr_{j}(p)$ be the trace of the Hecke operator $T_p$ acting on the space of weight $j$ holomorphic cusp forms on the full modular group.  Then for $j \neq 0$ and $p > 3$,
\begin{equation}
- \frac{1}{p-1} Q(p^j) = Tr_{j+2}(p).
\end{equation}
\end{myprop}
Remarks.  Here $p^{j/2} \lambda_{a,b}(p^j)$ is an integer.  The scaling factor $(p-1)^{-1}$ naturally arises because $|\lambda(p)|$ is fixed under quadratic twists (of which there are typically $p-1$).  If we define the normalized trace $Tr_{j}^*(n)$ via $Tr_{j}(n) = n^{(j-1)/2} Tr_{j}^*(n)$, then the proposition reads
\begin{equation}
Q^*(p^j) = - \frac{p-1}{p^{3/2}} Tr_{j+2}^*(p)
\end{equation}
for $j > 0$.  Of course $Q^*(1) = 1$.
\begin{proof}
The proof of this result is implicitly contained in \cite{Birch}.  Birch computes the related sum
\begin{equation}
p^{j/2} \left( \sum_a \sum_b \lambda_{a,b}(p) \right)^j.
\end{equation}
It is actually simpler to work with \eqref{eq:Q}.  By modifying Birch's arguments we easily arrive at
\begin{equation}
\frac{1}{p-1} Q(p^j) = 1 + \half \sum_{t^2 < 4p} U_j\left(\frac{t}{2\sqrt{p}}\right) V_w(t^2 - 4p),
\end{equation}
where $V_w(D)$ is the number of classes of binary quadratic forms with discriminant $D < 0$, divided by $3$ if $D=-3$ and divided by $2$ if $D=-4$.  This should be compared with (3) of \cite{Birch}.  The proof is completed by comparing the above equation with the Eichler-Selberg Trace Formula ((4.5) of \cite{Selberg}).
\end{proof}

\subsection{The general case}
Now we may easily prove the following
\begin{myprop}  
\label{prop:Q*formula}
Let $Q^*$ be given by \eqref{eq:Q*}.  
Then for $p > 3$ and even $f = \sum e_i \neq 0$ 
\begin{equation}
\label{eq:Q*formula}
Q^*(p^{e_1}, \ldots, p^{e_k}) 
= 
c_0 \frac{p-1}{p}
 -\sum_{l \geq 1} c_l \left( \frac{p-1}{p^{3/2}} Tr_{l+2}^*(p) + \frac{p-1}{p^2} p^{-l/2} \right) +  \frac{p-1}{p^2} p^{-f/2}
\end{equation}
holds. 
In addition, the left hand side above is zero if $f$ is odd or $p=2$ and the right hand side is $1 -p^{-2}$ if $f=0$.
\end{myprop}
Recall that the $c_l = c_l(e_1, \dots, e_k)$ are defined by \eqref{eq:Ukortho} and satisfy \eqref{eq:cr}.
\begin{proof}
To see that the left hand side is zero when $f$ is odd, simply apply the change of variables $a = d^2a'$, $b = d^3 b'$ where $d$ is a quadratic nonresidue $\pmod{p}$, and notice that the same sum is obtained except multiplied by $-1$.

The right hand side is easily seen to be $1 - p^{-2}$ when $f=0$, since $c_l = 0$ for $l > 0$, and $c_0 = 1$.

We may now assume $f \neq 0$ is even.  To begin, split the summation into two pieces as follows
\begin{equation}
Q^*(p^{e_1}, \ldots, p^{e_k}) =\frac{1}{p^2} \sum_{\substack{a, b \shortmod{p} \\ (p, \Delta) = 1}} \lambda_{a, b}(p^{e_1}) \cdots \lambda_{a, b}(p^{e_k}) + \frac{1}{p^2}\sum_{\substack{a, b \shortmod{p} \\ p |\Delta}} \lambda_{a, b}(p^{e_1}) \cdots \lambda_{a, b}(p^{e_k}),
\end{equation}
where $\Delta = -16(4a^3 + 27b^2)$ of course.
Using the Hecke relations and the definition of the $c_l$ given by \eqref{eq:Ukortho}, we have
\begin{align}
Q^*(p^{e_1}, \ldots, p^{e_k}) & = \frac{1}{p^2} \sum_{l} c_l \sum_{\substack{a, b \shortmod{p} \\ (p, \Delta) = 1}} \lambda_{a, b}(p^{l}) +  \frac{1}{p^2} \sum_{\substack{a, b \shortmod{p} \\ p |\Delta}} \lambda_{a, b}^f(p).
\end{align}
In the first sum we separate the term $l = 0$ and remove the condition $(p, \Delta) = 1$ to obtain
\begin{equation}
\frac{1}{p^2}  \sum_{\substack{a, b \shortmod{p} \\ (p, \Delta) = 1}} 1 + \sum_{l \geq 1} c_l Q^*(p^l) -\sum_{l \geq 1} c_l \frac{1}{p^2}  \sum_{\substack{a, b \shortmod{p} \\ p| \Delta}} \lambda_{a,b}(p^l) +  \frac{1}{p^2} \sum_{\substack{a, b \shortmod{p} \\ p |\Delta}} \lambda_{a, b}^f(p).
\end{equation}
One can parameterize all pairs $(a,b) \in \mf_p^2$ such that $\Delta \equiv 0 \pmod{p}$ by $a = -3c^2$, $b = 2c^3$, where $c$ runs modulo $p$.  In particular, there are $p$ such pairs, and hence there are $p(p-1)$ pairs such that $(p, \Delta) = 1$.
If $p | \Delta$ and $l$ is even then we claim that $\lambda_{a,b}(p^l) = p^{-l/2}$ unless $p|(a,b)$ or $p =3$, in which case $\lambda_{a,b}(p^l) = 0$.  To see this, note that $\lambda_{a,b}(p^l) = \lambda_{a,b}^l(p)$ for $p | \Delta$, and if $a=-3c^2$, $b=2c^3$, then
\begin{multline}
\lambda_{a,b}(p) = -\frac{1}{\sqrt{p}} \sum_{x \shortmod{p}} \left(\frac{x^3 -3c^2x + 2c^3}{p}\right) = -\frac{1}{\sqrt{p}} \sum_{x \shortmod{p}} \left(\frac{(x-c)^2(x+2c)}{p}\right) \\
= -\frac{1}{\sqrt{p}} \sum_{x\neq c} \left(\frac{x+2c}{p}\right)
= \frac{1}{\sqrt{p}}\left(\frac{3c}{p}\right).
\end{multline}

Hence we obtain
\begin{align}
Q^*(p^{e_1}, \ldots, p^{e_k}) & = \frac{p-1}{p} c_0 + \sum_{l \geq 1} c_l Q^*(p^l) - \frac{(p-1)}{p^2} \sum_{\substack{l \geq 1}} c_lp^{-l/2} +  \frac{(p-1)}{p^2} p^{-f/2},
\end{align}
for $p > 3$.
Applying Proposition \ref{prop:Q} completes the proof.
\end{proof}
\begin{mycoro}
\label{coro:Q*approx}
We have
\begin{equation}
Q^*(p,1,\dots, 1) = 0,
\end{equation}
and
\begin{equation}
Q^*(p,p,\dots, 1) = 1 - p^{-1}.
\end{equation}
\end{mycoro}
\begin{proof}
The former assertion was already stated in Proposition \ref{prop:Q*formula}.  As for the latter, we compute that $c_0(1,1,0,\dots, 0) = c_2(1,1,0,\dots, 0)=1$ since $U_1(x)^2 = 4x^2 = 1 + (4x^2 - 1) = U_0(x) + U_2(x)$.
\end{proof}

Now we can compute the arithmetical factor with
\begin{myprop}
\label{prop:Ak}
Let $H$ be given by \eqref{eq:Hdef1}.  Then
\begin{multline}
\label{eq:H}
H(z_1, \ldots, z_k) 
= \prod_{\substack{p|3q}} \prod_{j=1}^k \left( 1 - \frac{\lambda_{r,t}(p)}{p^{\half + z_j}} + \frac{1}{p^{1 + 2z_j}} \right)^{-1} \\
\prod_{p \notdiv 6q} \left\{1 + \left(1 - \frac{1}{p} \right)\left(1 - \frac{1}{p^{10}} \right)^{-1}
\left[
\left(-1 + \int \prod_{j=1}^k \left(1 - \frac{2 \cos \theta}{p^{\half + z_j}} + \frac{1}{p^{1 + 2z_j}}\right)^{-1} d\mu_{ST} \right) \right. \right.
\\
- \frac{1}{p^{1/2}} \sum_{l} Tr_{l+2}^*(p) \int U_l(\cos \theta) \prod_{j=1}^k \left(1 - \frac{2 \cos \theta}{p^{\half + z_j}} + \frac{1}{p^{1 + 2z_j}}\right)^{-1} d\mu_{ST}
\\
- \frac{p+1}{p^2} \int \left(-1 + \left(1 - \frac{2 \cos 2\theta}{p} + \frac{1}{p^2} \right)^{-1} \right) \prod_{j=1}^k \left(1 - \frac{2 \cos \theta}{p^{\half + z_j}} + \frac{1}{p^{1 + 2z_j}}\right)^{-1} d\mu_{ST}
\\
\left. \left. + \frac{1}{p} \left(- 1 + \half \left(\prod_{i=1}^k \left(1 - p^{-1 - z_i}\right)^{-1} + \prod_{i=1}^k \left(1 + p^{-1 - z_i}\right)^{-1}\right) \right) \right] \right\}.
\end{multline}
Furthermore, $H$ has the form
\begin{equation}
\label{eq:Ak}
H(z_1, \dots, z_k) = \left(\prod_{1 \leq i < j \leq k} \zeta(1 + z_i + z_j) \right) A_k(z_1, \dots, z_k),
\end{equation}
where $A_k$ is given by an Euler product that is uniformly convergent in the region $\text{Re}(z_i) \geq -\delta$, $1 \leq i \leq k$, for some $\delta > 0$.
\end{myprop}
\begin{proof}
Recall $H$ satisfies \eqref{eq:Hdef2}, so
\begin{multline}
H(z_1, \dots, z_k) = \prod_{p | 3q} \prod_{i=1}^{k}\left( 1- \frac{\lambda_{r,t}(p)}{p^{\half + z_i}} + \frac{1}{p^{1 + 2z_i}}\right)^{-1} 
\\
\prod_{p \nmid 6q}\left(1 + (1-p^{-10})^{-1} \sum_{\substack{e_1,\dots, e_k \\ e_1 + \dots + e_k > 0}}\frac{Q^*(p^{e_1}, \dots, p^{e_k})}{p^{e_1(\half + z_1) + \dots + e_k(\half + z_k)}} \right).
\end{multline}
We apply Proposition \ref{prop:Q*formula} and sum over the four terms given by \eqref{eq:Q*formula} separately.  We compute, using $c_0 = 0$ if $f$ is odd,
\begin{align}
 \sum_{\substack{e_1, \ldots, e_k \\ 2 |f \neq 0}} \frac{c_0(e_1, \ldots, e_k)}{p^{e_1(\half + z_1) + \ldots + e_k(\half + z_k)}} 
&= 
-1 +  \sum_{\substack{e_1, \ldots, e_k}} \int \frac{U_{e_1}(\cos \theta) \cdots U_{e_k}(\cos \theta) }{p^{e_1(\half + z_1) + \ldots + e_k(\half + z_k)}} d\mu_{ST}
\\
&= -1 + \int \prod_{j=1}^k \left(1 - \frac{2 \cos \theta}{p^{\half + z_j}} + \frac{1}{p^{1 + 2z_j}}\right)^{-1} d\mu_{ST}
\\
&= \sum_{1 \leq i < j \leq k} p^{-1 - z_i - z_j} + (\text{higher degree terms}),
\end{align}
this final estimation being easily seen using that $c_0(1,1,0,\dots,0) = 1$.
This term accounts for the Riemann zeta factors.

The term involving traces requires a computation as follows
\begin{multline}
\sum_{l \geq 1} Tr_{l+2}^*(p) \sum_{\substack{e_1, \ldots, e_k \\ 2 |f \neq 0}}  \frac{c_l(e_1, \ldots, e_k)}{p^{e_1(\half + z_1) + \ldots + e_k(\half + z_k)}} 
\\
= 
\sum_{l \geq 1} Tr_{l+2}^*(p) \sum_{\substack{e_1, \ldots, e_k}} \int \frac{U_{e_1}(\cos \theta) \cdots U_{e_k}(\cos \theta) }{p^{e_1(\half + z_1) + \ldots + e_k(\half + z_k)}} U_l(\cos \theta) d\mu_{ST}
\\
= \sum_{l} Tr_{l+2}^*(p) \int U_l(\cos \theta) \prod_{j=1}^k \left(1 - \frac{2 \cos \theta}{p^{\half + z_j}} + \frac{1}{p^{1 + 2z_j}}\right)^{-1} d\mu_{ST}.
\end{multline}
Since the traces are zero for $l+2 < 12$, this sum is uniformly convergent in a product of half-planes containing the origin.

The other term involving $c_l$ involves
\begin{align}
\sum_{l \geq 1} p^{-l/2} \sum_{\substack{e_1, \ldots, e_k \\ 2 |f \neq 0}}  \frac{c_l(e_1, \ldots, e_k)}{p^{e_1(\half + z_1) + \ldots + e_k(\half + z_k)}} 
= 
\sum_{l \geq 1} p^{-l} \sum_{\substack{e_1, \ldots, e_k}} \int \frac{U_{e_1}(\cos \theta) \cdots U_{e_k}(\cos \theta) }{p^{e_1(\half + z_1) + \ldots + e_k(\half + z_k)}} U_{2l}(\cos \theta) d\mu_{ST}
\\
= \sum_{l \geq 1} p^{-l} \int U_{2l}(\cos \theta) \prod_{j=1}^k \left(1 - \frac{2 \cos \theta}{p^{\half + z_j}} + \frac{1}{p^{1 + 2z_j}}\right)^{-1} d\mu_{ST} 
\\
= \int \left( -1 + \left(\frac{1 + \frac{1}{p}}{1 - \frac{2 \cos 2\theta}{p} + \frac{1}{p^2}} \right)\right) \prod_{j=1}^k \left(1 - \frac{2 \cos \theta}{p^{\half + z_j}} + \frac{1}{p^{1 + 2z_j}}\right)^{-1} d\mu_{ST},
\end{align}
where the summation over $l$ is executed using the identity
\begin{align}
\sum_{n=0}^{\infty} U_{2n}(x) t^{2n} = \frac{1 + t^2}{1 + 2t^2(1 - 2x^2) + t^4},
\end{align}
which is easily derived from \eqref{eq:Ukgenseries} by replacing $t$ with $-t$ and adding.

The final term is
\begin{align}
 \sum_{\substack{e_1, \ldots, e_k \\ 2 |f \neq 0}} \frac{1}{p^{e_1(1 + z_1) + \ldots + e_k(1 + z_k)}} 
= -1 + \half \sum_{\substack{e_1, \ldots, e_k}} \frac{1 + (-1)^{e_1 + \ldots + e_k}}{p^{e_1(1 + z_1) + \ldots + e_k(1 + z_k)}} 
\\
= - 1 + \half \left[\prod_{i=1}^k \left(1 - p^{-1 - z_i}\right)^{-1} + \prod_{i=1}^k \left(1 + p^{-1 - z_i}\right)^{-1}\right]
\end{align}
The proposition follows by appropriately summing the four terms above.
\end{proof}

To obtain a formula for $a_k$, we use that the leading coefficient of $P_k(N)$ is 
\begin{equation}
A_k(0,\ldots,0) 2^{k} \prod_{j=1}^{k-1} \frac{j!}{2j!} = A_k(0,\ldots, 0) 2^{k/2} \frac{ G(1+k) \sqrt{\Gamma(1+2k)}}{\sqrt{G(1 + 2k) \Gamma(1+k)}} =: a_k g_k,
\end{equation}
where $a_k = A_k(0, \ldots, 0)$.  This computation can be found in \cite{CFKRS}.  Then we compute
\begin{multline} 
\label{eq:ak}
a_k 
= 
\left(\frac{\phi(6q)}{6q}\right)^{k(k-1)/2} \prod_{\substack{p|3q }} \left( 1 - \frac{\lambda_{r,t}(p)}{p^{\half}} + \frac{1}{p^{}} \right)^{-k} \\
\cdot \prod_{p \notdiv 6q} \left(1 - \frac{1}{p} \right)^{k(k-1)/2} \left\{1 + \left(1 - \frac{1}{p} \right) \left( 1- \frac{1}{p^{10}} \right)^{-1} \left[ \left(-1 + \int V_p(\theta)^k d\mu_{ST} \right) \right. \right.
\\
- \frac{1}{p^{1/2}} \sum_{l} Tr_{l+2}^*(p) \int U_l(\cos \theta) V_p(\theta)^k d\mu_{ST}
\\
- \frac{p+1}{p^2} \int \left( -1 + \left(1 - \frac{2 \cos 2\theta}{p} + \frac{1}{p^2} \right)^{-1}\right) V_p(\theta)^k d\mu_{ST}
\\
\left. \left. + \frac{1}{p} \left(- 1 + \half \left( \left(1 - \frac{1}{p} \right)^{-k} + \left(1 + \frac{1}{p} \right)^{-k}\right) \right) \right] \right\},
\end{multline}
where
\begin{equation}
V_p(\theta) = \left(1 - \frac{2 \cos\theta}{p^{1/2}} + \frac{1}{p} \right)^{-1}.
\end{equation}

\section{The arithmetical factor for the positive rank family $\mathcal{F}'$}
\label{section:ak'}
The computation of the arithmetical factor for the family $\mathcal{F}'$ is more difficult than that for the family of all elliptic curves $\mathcal{F}$.  This author does not know an explicit formula similar to \eqref{eq:H}.  Nevertheless, we can compute $Q_{\square}^*(p^{e_1}, \ldots, p^{e_k})$ when $e_1 + \ldots + e_k$ is even.  The reason for this is that the change of variables $x \rightarrow rx$, $a \rightarrow r^2 a$ gives
\begin{equation}
\sum_{a \shortmod{p}} \lambda_{a, c} (p^l) = \left(\frac{r}{p}\right)^l \sum_{a \shortmod{p}} \lambda_{a, r^{-3} c}(p^l),
\end{equation}
so applying $b \rightarrow r^2 b$ gives for $l$ even that
\begin{equation}
\sum_a \sum_b \lambda_{a, b^2}(p^l) = \sum_a \sum_b \lambda_{a, r b^2}(p^l).
\end{equation}
We conclude that
\begin{equation}
\sum_a \sum_b \lambda_{a, b^2}(p^l) = \sum_a \sum_b \lambda_{a, b}(p^l),
\end{equation}

Using the same arguments as in the proof of Proposition \ref{prop:Q*formula} shows that 
\begin{equation}
Q_{\square}^*(p^{e_1}, \ldots, p^{e_k}) = Q^*(p^{e_1}, \ldots, p^{e_k})
\end{equation}
for $e_1 + \ldots + e_k$ even.

On the other hand, $Q_{\square}^*(p^l)$ is not so easily analysed for $l$ odd.  For $l = 1$ we compute directly
\begin{align}
Q_{\square}^*(p) = -p^{-5/2} \sum_a \sum_b \sum_x \left(\frac{x^3 + ax + b^2}{p} \right).
\end{align}
The summation over $a$ clearly vanishes unless $x=0$, in which case the summation over $a$ is $p$.  The summation over $b$ is $p-1$.  Hence
\begin{align}
Q_{\square}^*(p) = -p^{-1/2} + p^{-3/2}.
\end{align}

We now have enough information to deduce that \eqref{eq:H'} holds, as desired.  Obtaining a formula for $A_k'$ similar to the analogous formula (\ref{eq:H}--\ref{eq:Ak}) for $A_k$ requires a formula for $Q_{\square}^*(p^l)$ for $l$ odd.

For the application to the Riemann Hypothesis described in Section \ref{section:RH} it is relevant to know the region of convergence of $A_1'(\alpha)$.  We compute
\begin{align}
\sum_{e=0}^{\infty} \frac{Q_{\square}^*(p^e)}{p^{e(\half + \alpha)}} &= 1 -\frac{1}{p^{1 + \alpha}} + O(p^{-3(\half + \alpha)}) + O(p^{-2-\alpha}) 
\\
&= (1 - p^{-1-\alpha}) (1 + O(p^{-2(1+\alpha)}) + O(p^{-3(\half + \alpha)})),
\end{align}
so $A_1'(\alpha)$ is given by an absolutely and uniformly convergent Euler product in the region $\text{Re}(\alpha) > -\frac16$.  This region could be improved with better bounds on $Q_{\square}^*(p^e)$ for $e \geq 3$.  We have made no such attempt.

\section{Deriving Conjecture \ref{conj:rank2}}
\label{section:rank2derivation}
Our arguments follow those of \cite{CKRS} so we shall be brief.  For further elaboration of the method see their paper.  The idea is to consider the following ratio
\begin{equation}
R_{q,k} = \lim_{X \rightarrow \infty} \left( \sum_{\substack{E \in \mathcal{F}_{r,t}^+(X)}} L(1/2, E)^k \right) \biggl\slash
\left( \sum_{\substack{E \in \mathcal{F}_{r',t'}^+(X)}} L(1/2, E)^k \right).
\end{equation}
Arguments as in \cite{CKRS} lead to the conjecture that $R_q(X) \sim R_{q,-\half}$.
Our Conjecture \ref{conj:momentfirstterm} gives the asymptotic behavior of $R_{q,k}$ for general $k$.

The expectation is that
\begin{equation}
R_{q, k} = \frac{\prod_{\substack{p|q}} \left(1 - \frac{\lambda_{r,t}(p)}{p^{1/2}} + \frac{1}{p} \right)^{-k}}{\prod_{p | q} \left(1 - \frac{\lambda_{r',t'}(p)}{p^{1/2}} + \frac{1}{p} \right)^{-k}},
\end{equation}
since all other factors are independent of the choice or $r$ and $t \pmod{6q}$.  By random matrix theory considerations, taking $k = -1/2$ above gives the prediction.

\section{The Riemann Hypothesis}
\label{section:RH}
At the Riemann Hypothesis conference in 2002 at Courant sponsored by AIM, Iwaniec described an approach to RH using positive rank families of elliptic curves, such as the family $\mathcal{F}'$ considered in this paper.  Conrey \cite{C} has given a brief summary of the approach.  Here we show how to use the moment conjectures to frame the method.

The identity
\begin{equation}
\lim \frac{1}{|\mathcal{F}'(X)|} \sum_{E \in \mathcal{F}'(X)} \lambda_{E}(p) = -\frac{1}{\sqrt{p}} + O(p^{-3/2})
\end{equation}
and multiplicativity implies
\begin{equation}
\lim \frac{1}{|\mathcal{F}'(X)|} \sum_{E \in \mathcal{F}'(X)} \lambda_{E}(n) \approx \frac{\mu(n)}{\sqrt{n}},
\end{equation}
for $n$ squarefree.  The point is that the M\"{o}bius function can be obtained by averaging Dirichlet series coefficients of these positive rank elliptic curves.

Consider \eqref{eq:rankoneM} with $k=1$.  A quick calculation gives that
\begin{equation}
M(\alpha) = \frac{A'(\alpha)}{\zeta(1 + \alpha)} X_E^{-\half}(\thalf + \alpha),
\end{equation}
where $A'(\alpha)$ is the arithmetical factor that converges uniformly on compact subsets of the region $\text{Re}(\alpha) > -\frac16$, using the computation at the end of Section \ref{section:ak'}.

Thus we obtain
\begin{myconj}
\label{conj:RH}
Let $|\text{Re}(\alpha)| < \frac16$.  Then
\begin{equation}
\sum_{E \in \mathcal{F}'(X)} L(\thalf + \alpha, E) = \frac{A'(\alpha)}{\zeta(1 + \alpha)}  \sum_{E \in \mathcal{F}'(X)} (1 + O(N_E^{-\delta})).
\end{equation}
\end{myconj}
To deduce a quasi-Riemann Hypothesis, let $1 + \alpha$ be a nontrivial zero with $\text{Re}(\alpha) > -\frac16$).  The left hand side is obviously holomorphic at $\alpha$, but the right hand side is not.
\begin{mycoro}
Conjecture \ref{conj:RH} implies that the Riemann zeta function $\zeta(s)$ has no zeros for $\text{Re}(s) > \frac56$.
\end{mycoro}

\end{document}